\documentclass{article}

\usepackage{maa-monthly}

\usepackage{graphicx}
\usepackage{epstopdf}
\usepackage{color}
\usepackage{adjustbox}

\newtheorem{theorem}{\hspace{\parindent}
T{\scriptsize HEOREM}}[section]
\newtheorem{proposition}[theorem]
{\hspace{\parindent }P{\scriptsize ROPOSITION}}
\newtheorem{conjecture}[theorem]
{\hspace{\parindent }C{\scriptsize ONJECTURE}}

\def\n{\noindent}

\title{Unit Fractions in Norm-Euclidean Rings of Integers}
\author{Kyle Bradford and Eugen J.  Ionascu}

\begin{document}
\maketitle

\section{Introduction} \label{sec1}

The Erd\H{o}s-Straus conjecture became a topic of interest in the late 1940s and early 1950s \cite{pe1, ob, ro} and has since been the topic of many papers.  Richard Guy has a wonderful account of the progress on this work (see \cite{rg}).  In short, the conjecture asks to show that for every natural number $n \geq 2$, the Diophantine equation

\begin{equation}\label{esconj}
\frac{4}{n}=\frac{1}{a}+\frac{1}{b}+\frac{1}{c},
\end{equation}

\ 

\noindent has a solution $a$, $b$, $c\in \mathbb N$.  There have been many partial results about the nature of solutions to this equation.  Some people have used algebraic geometry techniques to give structure this problem (see \cite{cs}).  Many attempts use analytic number theory techniques to find mean and asymptotic results (see \cite{ec, et, dl, san1, san2, v, wweb1, yang}).  Some people have tried to look at decompositions of related fractions, such as $k \slash n$ for $k \geq 2$ (see \cite{aaa, ec, gm, rav, wweb2, wweb3}).  Some have tried computational methods (see \cite{sa}).  Many people have organized primes $p$  into two classes based off of the decompositions of $4 \slash p$  in hopes to find a pattern within each class (see \cite{bh, et, san1, san2}).  A well known method was developed by Mordell \cite{mo}  and many attempts use the techniques in that paper (see \cite{iw, san3, t, y}).  Some people attempt to find patterns in the field of fractions of the polynomial ring $\mathbb{Z}[x]$ instead of $\mathbb{Q}$ (see \cite{s}).  The conjecture that we make gives light to a completely different approach that one can take to find results about problems similar to Erd\H{o}s-Straus.

To begin we note that if $n \in \mathbb{Z}$  such that $|n| \geq 2$, then (\ref{esconj})  has a solution $a$, $b$, $c\in \mathbb Z$.  The following decompositions render this problem trivial:

\begin{equation} \label{esint}
\frac{4}{n} =
\begin{cases}
1 \slash k + 1 \slash k & \text{if $n = 2k$, $k \in \mathbb{Z}$  with $k \neq 0$} \\
1 \slash (k+1)+1 \slash ((k+1)(4k+1)) & \text{if $n=4k+3$, $k\in \mathbb{Z}$ with $k \neq -1$} \\
1 \slash k- 1 \slash (k(4k+1)) & \text{if $n=4k+1$, $k \in \mathbb{Z}$  with $k \neq 0$}.
\end{cases}
\end{equation}

\ 

Notice that for (\ref{esint}) we relaxed the restriction of having the values $a,b,c \in \mathbb{N}$, which is a specific cone within the integers.  With the Erd\H{o}s-Straus conjecture being unsolved for decades and this integer version easily solved as in (\ref{esint}), this illuminates a stark contrast in difficulty.  In this paper we find solutions to (\ref{esconj}) for a familiar class of algebraic number fields. We will highlight the Gaussian integers from this class to create our own conjecture.  The Gaussian integers $\mathbb{Z}[i]$  form a $\mathbb{Z}$-module with basis $\{ 1, i \}$.  Every prime in $\mathbb{Z}[i]$  has conjugates and associates.  If we can find a decomposition for a prime $n \in \mathbb{Z}[i]$  as in (\ref{esconj}), where $a,b,c \in \mathbb{Z}[i]$, then we will have a decomposition for all associates and conjugates of this prime.  It suffices to consider primes where both the real and imaginary parts are positive or, in other words, $n \in \mathbb{Z}[i]$  within the positive cone generated by the $\mathbb{Z}$-module basis $\{ 1, i \}$.  If we wanted to restrict the possible solutions to a specific cone within $\mathbb{Z}[i]$, then we would need to find $a,b,c \in \mathbb{Z}[i]$  within the positive and negative cone, or simply cone, generated by the $\mathbb{Z}$-module basis $\{ 1,i \}$.  The following conjecture is the analogue of the natural number Erd\H{o}s-Straus conjecture.

\begin{conjecture}{(Bradford-Ionascu)} \label{biconj}
Let ${\mathcal E}:=\{ 0, 1, i, 1+i \}$.  For $n \in \mathbb{Z}[i] \backslash {\mathcal E}$  with the real and imaginary part of $n$ nonnegative, (\ref{esconj}) has a solution $a, b, c \in \mathbb{Z}[i]$  such that the real and imaginary parts of $a,b$  and $c$ are either both nonnegative or both nonpositive.
\end{conjecture}

\begin{figure} \label{image1}
\includegraphics[scale = 0.6, trim = {2.5cm 21.5cm 0 0}]{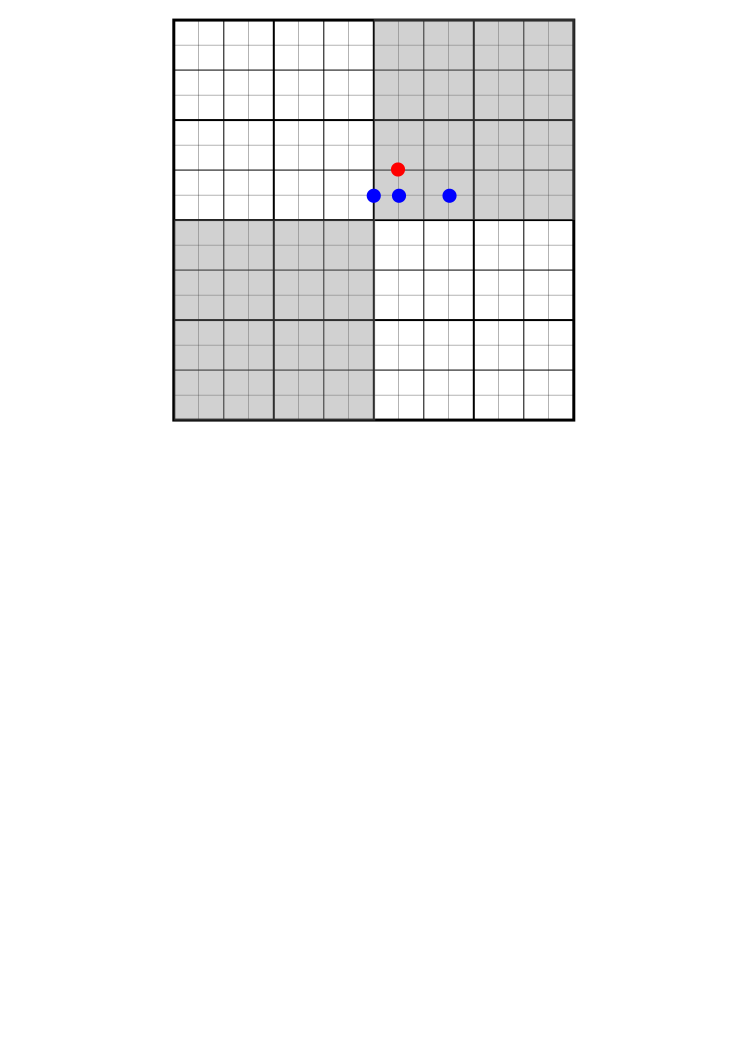}
\caption{This graph represents $\mathbb{C}$.  The red dot represents $1+2i$  and the blue dots are the values $x,y,z \in \mathbb{Z}$  that satisfy (\ref{esconj})  for $1+2i$}
\end{figure}

Figure \ref{image1}.  shows an example of (\ref{esconj}) having a solution $a,b,c \in \mathbb{Z}[i]$  when $n = 1+2i$.  Here 

\begin{equation} \label{example1} 
\frac{4}{1+2i} = \frac{1}{i} + \frac{1}{1+i} + \frac{1}{3+i}.
\end{equation}

\ 

Notice that all of the Gaussian integers in the denominators of the unit fractions in (\ref{example1}) are in the appropriate region for our conjecture, which is highlighted as the gray region in figure \ref{image1}.  However, the restriction of the solution to (\ref{esconj}) for $a,b,c \in \mathbb{Z}[i]$  to this cone introduces complications.  For example, if we have a solution to (\ref{esconj}) for $n \in \mathbb{N}$, then we necessarily have a solution of (\ref{esconj})  for $nm$  for any $m \in \mathbb{N}$.  In other words, the original Erd\H{o}s-Straus conjecture reduces to finding a solution to (\ref{esconj})  for prime natural numbers.  This is not the case for our conjecture.  For example $1+2i$  is prime in $\mathbb{Z}[i]$, which has a unique solution outlined in (\ref{example1}), and $1+i$  is a prime number in $\mathcal{E}$.  We see that $3+i = (-i)(1+i)(1+2i)$, yet we see that

\begin{equation} \label{example2}
\begin{split}
\frac{4}{3+i} &= \frac{1}{(-i)(1+i)} \cdot \frac{4}{1+2i} \\
&= \frac{1}{(-i)(1+i)} \cdot \left( \frac{1}{i} + \frac{1}{1+i} + \frac{1}{3+i} \right) \\
&= \frac{1}{1+i} + \frac{1}{2} + \frac{1}{4-2i}
\end{split}
\end{equation}

\ 

\noindent is the only decomposition possible when decomposing one of the prime factors of $3+i$.  This does not, however, imply that our conjecture does not hold.  It means that our conjecture does not reduce to finding a solution to (\ref{esconj}) for prime Gaussian integers outside of $\mathcal{E}$.  We see that 

\begin{equation} \label{example3}
\frac{4}{3+i} = \frac{1}{1} + \frac{1}{1+3i} + \frac{1}{5+5i}. 
\end{equation}

\ 

Before we attempt to show that (\ref{esconj}) has a solution for all Gaussian integers in this restricted cone, except for those in $\mathcal{E}$, we first seek to show that (\ref{esconj})  has a solution outside of an exceptional set for all Gaussian integers. The Gaussian integers form a unique factorization domain, so without a restriction to a cone, this will reduce to finding a solution of (\ref{esconj}) for prime Gaussian integers outside of an exceptional set.  Finding a solution of (\ref{esconj}) for all Gaussian integers will add a foundation to our conjecture, but the difficultly of finding solutions for these two should be vastly different as in the relaxation of finding integer solutions to (\ref{esconj}) rather than natural number solutions.  Also we address at this point that the Gaussian integers arise from a ring extension.  There are many extensions for which this process works and this paper discusses solutions to (\ref{esconj}) when those extensions are of degree two and norm-Euclidean.

\section{Main Result} \label{sec2}

We would like to determine whether (\ref{esconj}) having a solution in a ring with identity, outside of a finite, exceptional set, is a consequence of unique factorization or whether it requires more structure.  Finding solutions in general rings is difficult so we begin by considering the ring of integers for quadratic fields.  It is still unclear which rings of integers have unique factorization, but the norm-Euclidean quadratic fields have been fully classified \cite{fl}.  These fields are $\mathbb{Q}(\sqrt{d})$  where $d$  takes values $$ -11,-7, -3, -2, -1, 2, 3, 5, 6, 7, 11, 13, 17, 19, 21, 29, 33, 37, 41, 57, 73.$$

\ 

The rings of integers for quadratic fields have been thoroughly studied.  We will use the notation $\mathbb{D}[d]$  to represent the ring of integers for the quadratic field $\mathbb{Q}(\sqrt{d})$.  We can cite \cite{dam} to argue that the proof of the following is an elementary homework problem in algebraic number theory: 
\begin{equation} \label{ringofint}
\mathbb{D}[d]=
\begin{cases}
\mathbb{Z}[\sqrt{d}] &\text{if $d\equiv 3$ (mod $4$)}  \\
\mathbb{Z}[\frac{1+\sqrt{d}}{2}] &\text{if $d\equiv 1$ (mod $4$)}.
\end{cases}
\end{equation}  

\ 

In the introduction we showed that (\ref{esconj}) has a solution for integers values of $n$  when $|n| \geq 2$, however, the natural number solutions remain elusive.  We also mentioned that showing (\ref{esconj}) has a solution for Gaussian integers provides a foundation for proving conjecture \ref{biconj}.  The following theorem is the main result of this paper and it provides a foundation for making conjectures about solutions in restricted regions.  This result will prove that a sufficient condition for (\ref{esconj}) having a solution in $\mathbb{D}[d]$  is having $\mathbb{Q}(\sqrt{d})$  be norm-Euclidean.

\begin{theorem} \label{main}
Let $\mathbb{Q}(\sqrt{d})$  be a norm-Euclidean quadratic field and let $\mathbb{D}[d]$  be its ring of integers.  Letting $\mathcal{E}_{d}$  be a finite exceptional set, (\ref{esconj}) has a solution $a,b,c \in \mathbb{D}[d]$ for every $n \in \mathbb{D}[d] \backslash \mathcal{E}_{d}$.
\end{theorem}

This is not to say that $(\ref{esconj})$  does not have solutions in general for the rings of integers of quadratic fields that are not norm-Euclidean.  We can highlight this with the following decomposition in $\mathbb{Z}[\omega]$  where $\omega = (1 \slash 2) + (\sqrt{69} \slash 2)$: 

\begin{equation} \label{example4}
4 = \frac{1}{1710 + 468 \omega} + \frac{1}{2178 - 468 \omega} .
\end{equation}

\ 

It is well-known that the ring of integers for a quadratic field $\mathbb{Q}(\sqrt{d})$  will be a unique factorization domain if it has class number $1$.  Determining the values of $d \geq 0$  so that $\mathbb{D}[d]$ has class number $1$  is an open problem whereas it is well-established that the only possible values of $d \leq 0$  are those mentioned already for norm-Euclidean quadratic fields as well as the following: $$-19, -43, -67, -163.$$

\ 

Although we include no proofs in this paper, we also want to suggest there are likely decompositions in these cases as well.  We also conjecture that a sufficient condition for (\ref{esconj}) having a solution in the ring of integers for a quadratic field is having $\mathbb{D}[d]$  be a unique factorization domain. 

The rest of the paper is organized as follows:  in section \ref{sec3} we find decompositions for the ring of integers for norm-Euclidean quadratic fields when $d \geq 0$  and in section \ref{sec4}  we find decompositions when $d \leq 0$.

\section{Positive values} \label{sec3}

In this section we are interested in finding solutions to (\ref{esconj}) for the rings of integers $\mathbb{D}[d]$, where 

\begin{equation} \label{positivenumbers}
d \in \{ 2, 3, 5, 6, 7, 11, 13, 17, 19, 21, 29, 33, 37, 41, 57, 73 \}.
\end{equation}

\ 

It is quite interesting and somewhat unexpected that we have a rather trivial situation in each of these cases.

\begin{theorem}\label{squarerootof2} For every $n\in \mathbb{D}[d] \setminus \{0\}$ there exist
 $a$, $b$ in $\mathbb{D}[d]$ such that 
 
\begin{equation} \label{twounit}
\frac{4}{n}=\frac{1}{a}+\frac{1}{b}.
\end{equation}

\end{theorem}

\begin{proof}
\ 

\n The proof of this statement follows from the following identities:

$$ 4 = \frac{1}{-4+3\sqrt{2}} + \frac{1}{-4-3\sqrt{2}}, $$

$$ 4 = \frac{1}{2+\sqrt{3}} + \frac{1}{2-\sqrt{3}}, $$

$$ 4 = \frac{1}{-3+2\left( \frac{1}{2} + \frac{\sqrt{5}}{2} \right)} + \frac{1}{-1-2 \left( \frac{1}{2} + \frac{\sqrt{5}}{2} \right)}, $$

$$ 4 = \frac{1}{12+5\sqrt{6}} + \frac{1}{12-5\sqrt{6}}, $$

$$ 4 = \frac{1}{32+12\sqrt{7}} + \frac{1}{32-12\sqrt{7}}, $$

$$ 4 = \frac{1}{50+15\sqrt{11}} + \frac{1}{50-15\sqrt{11}}, $$

$$ 4 = \frac{1}{-207+90\left( \frac{1}{2} + \frac{\sqrt{13}}{2} \right)} + \frac{1}{-117-90 \left( \frac{1}{2} + \frac{\sqrt{13}}{2} \right)}, $$

$$ 4 = \frac{1}{-10+4\left( \frac{1}{2} + \frac{\sqrt{17}}{2} \right)} + \frac{1}{-6-4 \left( \frac{1}{2} + \frac{\sqrt{17}}{2} \right)}, $$

$$ 4 = \frac{1}{1} + \frac{1}{57+13\sqrt{19}} + \frac{1}{57-13\sqrt{19}}, $$

$$ 4 = \frac{1}{11+6\left( \frac{1}{2} + \frac{\sqrt{21}}{2} \right)} + \frac{1}{17-6 \left( \frac{1}{2} + \frac{\sqrt{21}}{2} \right)}, $$

$$ 4 =  \frac{1}{-2905 + 910 \left( \frac{1}{2} + \frac{\sqrt{29}}{2} \right)} + \frac{1}{-1995-910 \left( \frac{1}{2} + \frac{\sqrt{29}}{2} \right)}, $$

$$ 4 = \frac{1}{5+2\left( \frac{1}{2} + \frac{\sqrt{33}}{2} \right)} + \frac{1}{7-2 \left( \frac{1}{2} + \frac{\sqrt{33}}{2} \right)}, $$

$$ 4 = \frac{1}{-21+6\left( \frac{1}{2} + \frac{\sqrt{37}}{2} \right)} + \frac{1}{-15-6 \left( \frac{1}{2} + \frac{\sqrt{37}}{2} \right)}, $$

$$4=\frac{1}{-592+160\left( \frac{1}{2} + \frac{\sqrt{41}}{2} \right)}+\frac{1}{ -432-160\left( \frac{1}{2} + \frac{\sqrt{41}}{2} \right)},$$

$$4 = \frac{1}{33+10\left( \frac{1}{2} + \frac{\sqrt{57}}{2} \right)} + \frac{1}{43-10 \left( \frac{1}{2} + \frac{\sqrt{57}}{2} \right)}, $$

$$4=\frac{1}{-637062+133500\left( \frac{1}{2} + \frac{\sqrt{73}}{2} \right)}+\frac{1}{-503562-133500\left( \frac{1}{2} + \frac{\sqrt{73}}{2} \right)}.$$
\end{proof}

We point out that for the $d \geq 0$  mentioned above, the pattern appears to be that there exists $a+b\sqrt{d} \in \mathbb{D}[d] \backslash \{ 0 \}$ so that 

\begin{equation} \label{pell1} 
4 = \frac{1}{a+b\sqrt{d}} + \frac{1}{a-b\sqrt{d}}.
\end{equation}

\ 

This can be rewritten to suggest that for the given $d \geq 0$  there exist $a,b \in \mathbb{Z}$  such that 

\begin{equation} \label{pell2} 
(4a -1)^2 - d(4b)^{2} = 1.
\end{equation}  

\ 

If we relabel $x = 4a-1$  and $y = 4b$  we can see that we are looking for specific solutions to Pell's equation (see \cite{ne})

\begin{equation} \label{fullpell}
x^2 - dy^2 =1.
\end{equation}

\ 

It is not difficult to show that $4$  can be decomposed as in (\ref{twounit}) for all quadratic fields $\mathbb{D}[d]$  for which $d$  is a squarefree, positive integer.  

\section{Negative values} \label{sec4}

In this section we are interested in solving (\ref{esconj}) for rings of integers of norm-Euclidean quadratic fields $\mathbb{Q}(\sqrt{d})$  for which 

\begin{equation} \label{negativenumbers}
d \in \{-1,-2,-3,-7,-11\}.
\end{equation}

\ 

Notice that all of these fields are subsets of $\mathbb{C}$.  Much of the methodology in finding decompositions, as in (\ref{esconj}), for the rings in this section is the same, however, each ring brings its own complications.  To simplify this as much as possible we introduce some propositions that will be used in every scenario.  We also define some functions that will simplify our notation and make it easy to identify the general pattern to the decompositions.  

The first step in the every possible scenario of $d$  will be the same.  If we take any number $n \in \mathbb{D}[d]$  and divide it by $4$, we can consider the remainder and find our first unit fraction.  For example, if there exists $m,r \in \mathbb{D}[d]$  with $m \neq 0$  so that $n = 4m + r$, then we can write 

\begin{equation} \label{initialdecomp}
\frac{4}{n} = \frac{1}{m} - \frac{r}{nm}.
\end{equation}

\ 

The following proposition explains the nature of $r$  in every scenario and it will allow us to make initial statements about which values of $r$  will make $n$  a prime number.

\begin{proposition} \label{supporting6}
If we take any number in $\mathbb{D}[d]$  and divide it by $4$  we have sixteen possible remainders.  Expressing $\mathbb{D}[d] = \mathbb{Z}[\omega]$  where $\omega$  is defined as in (\ref{ringofint}), we see that the remainders will be $m+n\omega$  where $m,n \in \{ -2, -1, 0, 1, 2\}$  and $-2$  and $2$  can be interchanged.  \\

\noindent Furthermore, letting $x+y\omega = 4(a+b\omega) + (m+n\omega)$  with $m,n \in \{ -2,-1,0,1,2 \}$  and $m+n\omega$  be a multiple of a prime divisor of $2$, we have that $x+y\omega$  is not a prime number unless $a+b\omega = 0$  and $m+n$  is an associate of a prime divisor of $2$.
\end{proposition}

\begin{proof}
\ 

\noindent First notice that $\mathbb{D}[d] \simeq \mathbb{Z} \times \mathbb{Z}$.  \\

\noindent Denoting $(4)$  as the ideal generated by $4$  we see that $(4) \times (4)$  is an ideal in  $\mathbb{Z} \times \mathbb{Z}$. \\

\noindent Letting $\phi: \mathbb{Z} \times \mathbb{Z} \rightarrow \mathbb{Z}_{4} \times \mathbb{Z}_{4}$  be a ring homomorphism such that $\phi(1,0) = (1,0)$  and $\phi(0,1) = (0,1)$  we see that $\ker (\phi) = (4) \times (4)$. \\

\noindent Dividing a number in $\mathbb{D}[d]$  by $4$  and considering its remainder can be understood by taking the quotient of $\mathbb{D}[d]$  by $(4) \times (4)$, so we use the first isomorphism theorem to say that $\mathbb{Z} \times \mathbb{Z} \slash \ker (\phi) \simeq \mathbb{Z}_{4} \times \mathbb{Z}_{4}$. \\

\noindent The cardinality of this ring is $16$.  This implies that there are $16$  possible remainders after dividing a number in $\mathbb{D}[d]$  by $4$. \\

\noindent Because the possible remainders have a ring structure that is isomorphic to $\mathbb{Z}_{4} \times \mathbb{Z}_{4}$, we also see that the possible values of these remainders have coordinates that are elements of $\{ -2, -1, 0, 1, 2 \}$  where $-2$  and $2$  can be interchanged. \\

\noindent The second part of the proof is trivial.  If we let $m+n\omega$  be a multiple of a prime divisor of $2$, we see that $4$  is also a multiple of that same prime divisor.  Unless $m+n\omega$  is an associate of the prime divisor of $2$  and $a+b\omega = 0$, we can factor out the prime divisor from $m+n\omega$  and $4$  to find that $x+y\omega$  is a composite number. 
\end{proof}

We define a function $p: \mathbb{D}[d] \times \mathbb{D}[d] \rightarrow \mathbb{D}[d]$  by $p(a,b) = 4a + b$.  This function helps us account for the $16$ remainder scenarios from proposition \ref{supporting6}.  The value of $b$  will tell us which coset we are using and the different remainders require different techniques to find the decomposition as in (\ref{esconj}).  There are some remainders that use the same method for finding a decomposition.  The following two propositions reduce the number of remainder scenarios to decompose by using some symmetry within the rings $\mathbb{D}[d]$.

\begin{proposition} \label{supporting1}
Suppose that $b \in \mathbb{D}[d] \backslash \{ 0 \}$.  If there exists a decomposition as in (\ref{esconj}) for $4 \slash p(a,b)$  for all $a \in \mathbb{D}[d]$, then there exists a decomposition as in (\ref{esconj}) for $4 \slash p(a, ub)$  for all $a \in \mathbb{D}[d]$  and units $u \in \mathbb{D}[d]$.
\end{proposition}

\begin{proof}
\ 

\noindent Let $b \in \mathbb{D}[d] \backslash \{ 0 \}$.  Let $u_{1}, u_{2} \in \mathbb{D}[d]$  be units such that $u_{1} u_{2} = 1$.  For all $a \in \mathbb{D}[d]$  suppose that there exists $x,y,z \in \mathbb{D}[d]$  such that  $$ \frac{4}{p(a,b)} = \frac{1}{x} + \frac{1}{y} + \frac{1}{z}.$$

\

\noindent This implies that there exists $x', y', z' \in \mathbb{D}[d]$  for any $a \in \mathbb{D}[d]$ so that $$ \frac{4}{p(u_{2} a, b)} = \frac{1}{x'} + \frac{1}{y'} + \frac{1}{z'}.$$

\ 

\noindent Notice then that 

\begin{align*}
\frac{4}{p(a,u_{1} b)} &= \frac{4}{p(u_{1} u_{2} a, u_{1}b)} \\
&= \frac{4}{u_{1} p(u_{2} a,b)} \\
&= \frac{1}{u_{1}x'} + \frac{1}{u_{1}y'} + \frac{1}{u_{1}z'}.
\end{align*}
\end{proof}

As we mentioned earlier, all of the rings, $\mathbb{D}[d]$, discussed in this section are subsets of $\mathbb{C}$.   These rings are closed under conjugation.  The following proposition again reduces the number of remainder scenarios we need to consider through symmetry. 

\begin{proposition} \label{supporting2}
Suppose that $b \in \mathbb{D}[d] \backslash \{ 0 \}$.  If there exists a decomposition as in (\ref{esconj}) for $4 \slash p(a,b)$  for all $a \in \mathbb{D}[d]$, then there exists a decomposition as in (\ref{esconj}) for $4 \slash p(a, \bar{b})$  for all $a \in \mathbb{D}[d]$ .
\end{proposition}

\begin{proof}
\ 

\noindent Let $b \in \mathbb{D}[d] \backslash \{ 0 \}$.  For all $a \in \mathbb{D}[d]$  suppose there exists $x,y,z \in \mathbb{D}[d]$  such that $$ \frac{4}{p(a,b)} = \frac{1}{x} + \frac{1}{y} + \frac{1}{z}.$$

\noindent This implies that there exists $x',y',z' \in \mathbb{D}[d]$  for any $a \in \mathbb{D}[d]$  so that $$ \frac{4}{p(\bar{a},b)} = \frac{1}{x'} + \frac{1}{y'} + \frac{1}{z'}.$$

\noindent Notice then that 

\begin{align*}
\frac{4}{p(a,\bar{b})} &= \frac{4}{\overline{p(\bar{a},b)}} \\
&= \overline{\frac{4}{p(\bar{a},b)}} \\
&= \overline{\frac{1}{x'} + \frac{1}{y'} + \frac{1}{z'}} \\
&= \frac{1}{\overline{x'}} + \frac{1}{\overline{y'}} + \frac{1}{\overline{z'}}.
\end{align*}
\end{proof}

At this point we consider the remainder scenarios that exist after reductions through symmetry.  Some scenarios are shown to have decompositions rather easily while other scenarios require a more advanced method to find the decompositions.  For every ring, $\mathbb{D}[d]$, the methods used in the more complicated scenarios are roughly the same.  For example, we argued that finding an initial decomposition as in (\ref{initialdecomp}) would be the first step for finding the decomposition as in (\ref{esconj}).  For each remainder scenario, after the first division by $4$, the next step is to divide $m$  by $r$  and consider the possible remainders.  The following proposition tells us the nature of these remainders.  This proposition heavily uses the nature of a Euclidean domain to make its claim.

\begin{proposition} \label{supporting3}
If we let $x,n \in \mathbb{D}[d]$  be any numbers such that $|n|^{2}$  is odd, $|n| \neq 1$  and the complex (nonreal) component of $n$  is relatively prime from $|n|^{2}$, then there exists a number $q \in \mathbb{D}[d]$  and $r \in \mathbb{Z}$  such that $|r| \leq (|n|^{2} - 1) \slash 2$  and $x = nq + r.$
\end{proposition}

\begin{proof}
\ 

\noindent First note that $\mathbb{D}[d]$  is a Euclidean domain.  This means that for $x,n \in \mathbb{D}[d]$  there exist $q',r' \in \mathbb{D}[d]$  such that $x = nq' + r'$  where $|r'| < |n|$. \\

\noindent Also note that because $\mathbb{D}[d]$  is the ring of integers of a norm-Euclidean quadratic field, that $|n|^{2}$  will be an integer. \\ 

\noindent Our first goal is to show that there exists $m \in \mathbb{Z}$  so that $r' + mn$  has complex component that is a multiple of $|n|^{2}$. \\

\noindent Let the complex component of $r'$  be $a \in \mathbb{Z}$  and the nonreal component of $n$  be $b \in \mathbb{Z}$.  \\

\noindent Because $b$  and $|n|^{2}$  are relatively prime, we see that there exists $s,t \in \mathbb{Z}$  such that $1 = sb + t|n|^{2}$. \\

\noindent If we let $m = -as$  we see that 

\begin{align*}
a +mb &= a - asb \\
&= a - a(1 - t|n|^{2}) \\
&= a t |n|^{2}.
\end{align*}

\noindent This shows that the complex component of $r' +mn$  is a multiple of $|n|^{2}$.  So if the real component of $r'+mn$  is greater than $|n|^{2}$  when $m = -as$  we see that we can express it as $r + m' |n|^{2}$  where $r, m' \in \mathbb{Z}$  and $|r| < |n|^{2}$.  Define $k \in \mathbb{D}[d]$  as a number with real component $m'$  and complex component $at$. \\

\noindent We see now that 

\begin{align*}
x &= nq' + r' \\
&=n(q' - m) + (r' + mn) \\
&=n(q' - m) + (r + k |n|^{2}) \\
&= n(q' - m + k\bar{n}) + r.  
\end{align*}

\ 

\noindent If we let $q = (q' - m + k\bar{n})$, then we see that $x = nq + r$  where $r$  is an integer such that $|r| < |n|^{2}$. \\

\noindent Because $|n|^{2}$  is odd, we see that if $|r| > (|n|^{2} - 1) \slash 2$, then  $|n|^{2} - |r| \leq (|n|^{2} - 1) \slash 2$.  We can let $q = (q' - m + k\bar{n}) \pm \bar{n}$  in the appropriate scenario and rename $r$  so without loss of generality we can assume that $|r| \leq (|n|^{2} - 1) \slash 2$.
\end{proof}

Define the function for $n \in \mathbb{D}[d]$, $q_{n}: \mathbb{D}[d] \times \mathbb{D}[d] \rightarrow \mathbb{D}[d]$  such that $q_{n}(a,b) = na + b$.  We can use this function to reduce the amount of work in our method further through symmetry.  For example, we first found that if $n = 4m +r$, then we need to find decompositions as in (\ref{esconj}) for a few possible remainders $r$.  We can find decompositions quite easily for some values of $r$.  For other values of $r$  we must divide $m$  by $r$ to derive possible scenarios for these new remainders and we must find decompositions for all of these possible remainders.  The decompositions for some of these possible scenarios after the second division are redundant.  The following proposition accounts for the redundancies.

\begin{proposition} \label{supporting4}

Suppose that $n,r \in \mathbb{D}[d] \backslash \{ 0 \}$.  If there exists a decomposition as in (\ref{esconj}) for $4 \slash p(q_{n}(b,r),-n)$  for all $b \in \mathbb{D}[d]$, then there exists a decomposition as in (\ref{esconj}) for $4 \slash p(q_{n}(b,ur), -n)$  for all $b \in \mathbb{D}[d]$  and units $u \in \mathbb{D}[d]$.
\end{proposition}

\begin{proof}
\ 

\noindent Let $n,r \in \mathbb{D}[d] \backslash \{ 0 \}$.  Let $u_{1}, u_{2} \in \mathbb{D}[d]$  be units such that $u_{1}u_{2} = 1$.  For all $b \in \mathbb{D}[d]$  suppose that there exists $x,y,z \in \mathbb{D}[d]$  such that $$ \frac{4}{p(q_{n}(b,r), -n)} = \frac{1}{x} + \frac{1}{y} + \frac{1}{z}.$$

\ 

\noindent Proposition \ref{supporting1}  tells us that for all $b \in \mathbb{D}[d]$ there exists $x',y',z' \in \mathbb{D}[d]$  such that $$ \frac{4}{p(q_{n}(b,r), -u_{2}n)} = \frac{1}{x'} + \frac{1}{y'} + \frac{1}{z'}.$$ 

\ 

\noindent This implies that there exists $x'',y'',z'' \in \mathbb{D}[d]$  for any $b \in \mathbb{D}[d]$  so that $$ \frac{4}{p(q_{n}(u_{2}b,r),-u_{2}n)} = \frac{1}{x''} + \frac{1}{y''} + \frac{1}{z''}. $$

\ 

\noindent Notice then that

\begin{align*}
\frac{4}{p(q_{n}(b,u_{1}r),-n)} &= \frac{4}{p(q_{n}(u_{1}u_{2}b,u_{1}r), -u_{1}u_{2}n)} \\
&= \frac{4}{u_{1} p(q_{n}(u_{2}b, r), -u_{2}n)} \\
&= \frac{1}{u_{1}x''} + \frac{1}{u_{1}y''} + \frac{1}{u_{1}z''}.
\end{align*}
\end{proof}

At this point, each ring will use the properties intrinsic to the ring to find the decompositions.  We mentioned earlier that some decompositions were quite simple and the decompositions in the following proposition show that there are some simple decompositions that have the same basic pattern across all the possible rings in this section. 

\begin{proposition} \label{supporting5}
For $a,b \in \mathbb{D}[d] \backslash \{ 0 \}$ and any $n \in \mathbb{D}[d]$

\begin{align}
\frac{4}{p(a,-1)} &= \frac{1}{a} + \frac{1}{a \cdot p(a,-1)} \\
\frac{4}{p(q_{n}(b,1), -n)} &= \frac{1}{q_{n}(b,1)} + \frac{1}{p(q_{n}(b,1), -n) \cdot b} - \frac{1}{p(q_{n}(b,1), -n) \cdot q_{n}(b,1) \cdot b}.
\end{align}
\end{proposition}

\begin{proof}
\ 

\noindent First we see that

\begin{align*}
\frac{4}{p(a,-1)} &= \frac{4}{4a-1} \\
&= \frac{1}{a} + \frac{1}{a \cdot (4a-1)} \\
&= \frac{1}{a} + \frac{1}{a \cdot p(a,-1)}.
\end{align*}

\ 

\noindent Next we see that

\begin{align*}
\frac{4}{p(q_{n}(b,1),-n)} &= \frac{4}{4(nb+1) - n} \\
&= \frac{1}{nb+1} + \frac{n}{(4(nb+1) - n ) \cdot (nb+1)} \\
&= \frac{1}{nb+1} + \frac{nb +1 - 1}{(4(nb+1) -n) \cdot (nb+1) \cdot b} \\
&= \frac{1}{nb+1} + \frac{1}{(4(nb+1) -n) \cdot b} - \frac{1}{(4(nb+1) - n) \cdot (nb+1) \cdot b} \\
&=\frac{1}{q_{n}(b,1)} + \frac{1}{p(q_{n}(b,1), -n) \cdot b} - \frac{1}{p(q_{n}(b,1), -n) \cdot q_{n}(b,1) \cdot b}.
\end{align*}
\end{proof}

\begin{figure} 
\vspace{-5cm}
\begin{center} 
\adjustbox{trim={0.085\width} {0.0\height} {0.0\width} {0.0\height}, clip}{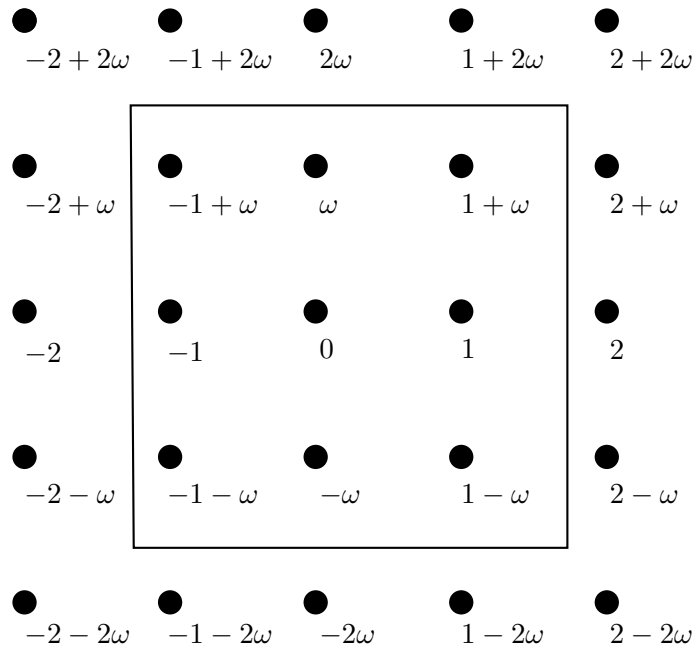}
\end{center}
\vspace{-15cm}
\caption{This represents the Gaussian integer lattice $\mathbb{Z}[\omega]$  when $\omega = i = \sqrt{-1}$.  Though it is not to scale, it can also represent the lattice for $\mathbb{Z}[\omega]$  when $\omega = \sqrt{-2}$.  In both cases, the lattice points inside the square are precisely the points in $\mathcal{E}_{-1}$  and $\mathcal{E}_{-2}$  respectively.}
\label{Square12}
\end{figure}

At this point we put into action our general methodology for finding decompositions.  While this method is similar for each value in (\ref{negativenumbers}), we will see that the decompositions for each ring is unique.  We begin by considering the Gaussian integers.  To make the notation similar for each ring we let $\omega = i = \sqrt{-1}$.  For $a+b\omega \in \mathbb{Z}[\omega]$, $|a+b\omega|^{2} = a^{2} + b^{2}$.  Let $\mathcal{E}_{-1} = \{ n \in \mathbb{Z}[\omega] : |n|^{2} \leq 2 \}$.

\begin{theorem}
There exists a decomposition similar to (\ref{esconj}) for every element in $\mathbb{D}[-1] \backslash \mathcal{E}_{-1}$.
\end{theorem}

\begin{proof}
\ 

\noindent If we divide the denominator by $4$, then proposition \ref{supporting6}  tells us that we have sixteen possible remainders.  Because all associates of the prime divisors of $2$  are elements of $\mathcal{E}_{-1}$, we see that eight remainders do not generate prime numbers.  These remainders are $0, 2,  2\omega , 2+2\omega, 1+\omega, 1-\omega, -1+\omega, -1-\omega$. \\

\noindent Due to propositions \ref{supporting1}  and \ref{supporting2}  we see that of the left-over remainders it suffices to find a decomposition for remainders $-1,-(1 - 2\omega)$.  This is because four remainders are associates of $-1$  and four remainders are associates of $-(1 - 2\omega)$.  \\

\noindent Proposition \ref{supporting5}  finds a decomposition when the remainder is $-1$. \\

\noindent Because $|1-2\omega|^{2} = 5$  and $(2,5) =1$, we can use proposition \ref{supporting3} to suggest that for $x+y\omega \in \mathbb{Z}[\omega]$  there exists $c+d\omega \in \mathbb{Z}[\omega]$  and $r \in \{ -2, -1, 0 , 1, 2 \}$  so that $x+y\omega = (1-2\omega)(c+d\omega) + r$.  Let $q_{1-2\omega} \left( (c+d\omega), r \right) = (1-2\omega)(c+d\omega) + r.$ \\

\noindent Notice that $p \left( q_{1-2\omega} \left( (c+d\omega), 0 \right), -(1-2\omega) \right)$  is not a prime number except when $c+d\omega = 0$, but in this case $$\frac{4}{-1+2\omega} = \frac{1}{\omega} + \frac{1}{-1+\omega} + \frac{1}{-3+\omega}.$$

\ 

\noindent Proposition \ref{supporting4} tells us that it suffices to find a solution for $r \in \{1,2 \}$. \\

\noindent Proposition \ref{supporting5}  finds a decomposition when $r=1$. \\

\noindent Define 
\begin{align*}
s_{1-2\omega,2} \left( q_{1-2\omega} \left( (c+d\omega), 2 \right) \right) &= \frac{q_{1-2\omega} \left( (c+d\omega), 2 \right) + \omega}{1-2\omega}\\
&= c+(d+1)\omega.
\end{align*}

\ 

\noindent If we write

\begin{align*}
p &= p \left( q_{1-2\omega} \left( (c+d\omega),2 \right), -(1-2\omega) \right) \\
q_{1-2\omega} &= q_{1-2\omega} \left( (c+d\omega),2 \right) \\
s_{1-2\omega,2} &= s_{1-2\omega,2} \left( q_{1-2\omega} \left( (c+d\omega),2 \right) \right)
\end{align*}

\ 

\noindent we see that

\begin{align*}
\frac{4}{p} &= \frac{1}{q_{1-2\omega}} + \frac{1-2\omega}{p \cdot q_{1-2\omega}} \\
&= \frac{1}{q_{1-2\omega}} + \frac{(1-2\omega) \cdot s_{1-2\omega,2}}{p \cdot q_{1-2\omega} \cdot s_{1-2\omega,2}} \\
&= \frac{1}{q_{1-2\omega}} + \frac{q_{1-2\omega} + \omega}{p \cdot q_{1-2\omega} \cdot s_{1-2\omega,2}} \\
&= \frac{1}{q_{1-2\omega}} + \frac{1}{p \cdot s_{1-2\omega,2}} - \frac{1}{\omega \cdot p \cdot q_{1-2\omega} \cdot s_{1-2\omega,2}}.
\end{align*}

\ 

We only need to mention the following decomposition because the prime of which it is a power was is $\mathcal{E}_{-1}$:

\begin{align*}
\frac{4}{(1+ \omega)^{2}} &= \frac{1}{\omega} + \frac{1}{2\omega} + \frac{1}{2\omega}.
\end{align*}
\end{proof}

Next we denote $\omega = \sqrt{-2}$.  For $a+b\omega \in \mathbb{Z}[\omega]$, $|a+b\omega|^{2} = a^{2} +2 b^{2}$.  Let $\mathcal{E}_{-2} = \{ n \in \mathbb{Z}[\omega] : |n|^{2} \leq 3 \}$.

\begin{theorem}
There exists a decomposition similar to (\ref{esconj}) for every element in $\mathbb{D}[-2] \backslash \mathcal{E}_{-2}$.
\end{theorem}

\begin{proof}
\ 

\noindent If we divide the denominator by $4$, proposition \ref{supporting6}  tells us that we have sixteen possible remainders.  Because all the associates of the prime divisors of $2$  are elements of $\mathcal{E}_{-2}$, we see that eight remainders do not generate prime numbers.These remainders are $0, \omega, 2\omega, -\omega, 2, 2 + \omega, 2 + 2\omega, 2 - \omega$.  \\

\noindent Due to propositions \ref{supporting1}  and \ref{supporting2}  we see that of the left-over remainders it suffices to find a decomposition for remainders $-1, -(1+\omega), -(1 + 2\omega)$. This is because two remainders are associates of $-1$, two remainders are associates of $-(1+\omega)$, two remainders are associates of the conjugate of $-(1+\omega)$  and two remainders are associates of $-(1 + 2\omega)$. \\

\noindent \noindent Proposition \ref{supporting5}  finds a decomposition when the remainder is $-1$. \\

\noindent We first find a decomposition for remainder $-(1+ \omega)$.  Because $|1+\omega|^{2} = 3$  and $(1,3) = 1$, we can use proposition \ref{supporting3}  to suggest that for $x+y\omega \in \mathbb{Z}[\omega]$  there exists $c+d\omega \in \mathbb{Z}[\omega]$  and $r \in \{ -1, 0 , 1 \}$  so that $x+y\omega = (1+\omega)(c+d\omega) + r$.  Let $q_{1+\omega} \left( (c+d\omega), r \right) = (1+\omega)(c+d\omega) + r.$ \\

\noindent Notice that $p \left( q_{1+\omega} \left( (c+d\omega), 0 \right), -(1+\omega) \right)$  is not a prime number except when $c+d\omega = 0$, but this prime number is an element of $\mathcal{E}_{-2}$. \\

\noindent Proposition \ref{supporting4} tells us that it suffices to find a solution for $r=1$  and proposition \ref{supporting5}  finds a decomposition when $r=1$. \\

\noindent Next we find a decomposition for remainder $-(1+2\omega)$.  Because $|1+2\omega|^{2}=9$  and $(2,9)=1$, we can use proposition \ref{supporting3}  again to suggest that for $x+y\omega \in \mathbb{Z}[\omega]$  there exists $c+d\omega \in \mathbb{Z}[\omega]$  and $r \in \{ -4, -3, -2, -1, 0 , 1, 2, 3, 4 \}$  so that $x+y\omega = (1+2\omega)(c+d\omega) +r$.  Let $q_{1+2\omega} \left( (c+d\omega), r \right) = (1+2\omega)(c+d\omega) + r.$ \\

\noindent Because $(1+2\omega) = -1 \cdot (1-\omega)^{2}$  and $3 = (1-\omega)\cdot (1 + \omega)$  we see \linebreak $p \left( q_{1+2\omega} \left( (c+d\omega), 0 \right), -(1+2\omega) \right)$  and $p \left( q_{1+2\omega} \left( (c+d\omega), \pm 3 \right), -(1+2\omega) \right)$  are not prime numbers. \\

\noindent Proposition \ref{supporting4} tells us that it suffices to find a solution for $r \in \{ 1, 2, 4 \}$. \\

\noindent Proposition \ref{supporting5}  finds a decomposition when $r=1$. \\

\noindent Define 

\begin{align*}
s_{1+2\omega,2} \left( q_{1+2\omega} \left( (c+d\omega), 2 \right) \right) &= \frac{p \left( q_{1+2\omega} \left( (c+d\omega), 2 \right), -(1+2\omega) \right) + 1}{1+2\omega} \\
&= 4(c+d\omega) - 2\omega 
\end{align*}
\begin{align*}
s_{1+2\omega,4} &\left( q_{1+2\omega} \left( (c+d\omega), 4 \right) \right) \\
&= \frac{ p \left( q_{1+2\omega} \left( (c+d\omega), 4 \right), -(1+2\omega) \right) \cdot q_{1+2\omega} \left( (c+d\omega), 4 \right) - 1}{1+2\omega} \\
&= (c+d\omega) \left( 4 \left( (1+2\omega)(c+d\omega) + 8 \right) - (1+2\omega) \right) + (3 - 14\omega).
\end{align*}

\ 

\noindent If we write 
\begin{align*}
p &= p \left( q_{1+2\omega} \left( (c+d\omega), 2 \right), -(1+2\omega) \right) \\
q_{1+2\omega} &= q_{1+2\omega} \left( (c+d\omega), 2 \right) \\  
s_{1+2\omega,2} &= s_{1+2\omega,2} \left( q_{1+2\omega} \left( (c+d\omega), 2 \right) \right)
\end{align*}

\ 

\noindent we see that 
\begin{align*} 
\frac{4}{p} &= \frac{1}{q_{1+2\omega}} + \frac{1+2\omega}{p \cdot q_{1+2\omega}} \\
&= \frac{1}{q_{1+2\omega}} + \frac{(1+2\omega) \cdot s_{1+2\omega,2} }{p \cdot q_{1+2\omega} \cdot s_{1+2\omega,2}} \\
&= \frac{1}{q_{1+2\omega}} + \frac{p + 1}{p \cdot q_{1+2\omega} \cdot s_{1+2\omega,2}} \\
&= \frac{1}{q_{1+2\omega}} + \frac{1}{q_{1+2\omega} \cdot s_{1+2\omega,2}} + \frac{1}{p \cdot q_{1+2\omega} \cdot s_{1+2\omega,2}}.
\end{align*}

\ 

\noindent If we write 
\begin{align*}
p &= p \left( q_{1+2\omega} \left( (c+d\omega), 4 \right), -(1+2\omega) \right) \\
q_{1+2\omega} &= q_{1+2\omega} \left( (c+d\omega), 4\right) \\  
s_{1+2\omega,4} &= s_{1+2\omega,4} \left( q_{1+2\omega} \left( (c+d\omega), 4 \right) \right)
\end{align*}

\ 

\noindent we see that 
\begin{align*} 
\frac{4}{p} &= \frac{1}{q_{1+2\omega}} + \frac{1+2\omega}{p \cdot q_{1+2\omega}} \\
&= \frac{1}{q_{1+2\omega}} + \frac{(1+2\omega) \cdot s_{1+2\omega,4} }{p \cdot q_{1+2\omega} \cdot s_{1+2\omega,4}} \\
&= \frac{1}{q_{1+2\omega}} + \frac{p \cdot q_{1+2\omega} - 1}{p \cdot q_{1+2\omega} \cdot s_{1+2\omega,4}} \\
&= \frac{1}{q_{1+2\omega}} + \frac{1}{ s_{1+2\omega,4}} - \frac{1}{p \cdot q_{1+2\omega} \cdot s_{1+2\omega,4}}.
\end{align*}

\ 

\noindent We only need to mention the following decompositions because the primes of which they are products of were in $\mathcal{E}_{-2}$:

\begin{align*}
\frac{4}{\omega^{2}} &= \frac{1}{-1} + \frac{1}{-2} + \frac{1}{-2} \\
\frac{4}{(1+\omega)^{2}} &= \frac{1}{\omega} + \frac{1}{-2+\omega} + \frac{1}{-1+2\omega} \\
\frac{4}{(1-\omega)(1+\omega)} &= \frac{1}{2} + \frac{1}{2} + \frac{1}{3} \\
\frac{4}{\omega(1+\omega)} &= \frac{1}{-1} + \frac{1}{\omega} + \frac{1}{-2+\omega}.
\end{align*}
\end{proof}
\ 

\begin{figure} 
\vspace{-2cm}
\adjustbox{trim={0.07\width} {0.0\height} {0.0\width} {0.0\height}, clip}{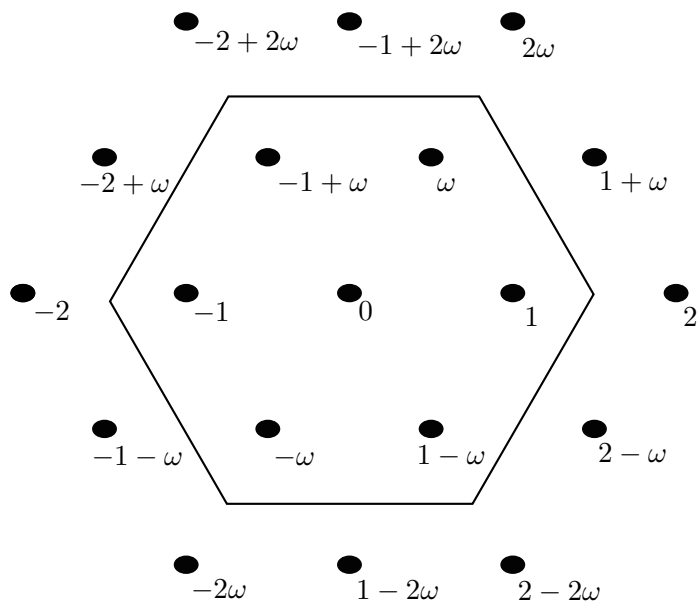}
\vspace{-18.5cm}
\caption{This represents the the Eisenstein integer lattice $\mathbb{Z}[\omega]$  when $\omega = (1 \slash 2) + (\sqrt{-3} \slash 2)$.  Though it is not to scale, it can also represent the lattice for $\mathbb{Z}[\omega]$  when $\omega = (1 \slash 2) + (\sqrt{-7} \slash 2)$.  In both cases, the lattice points inside the hexagon are precisely the points in $\mathcal{E}_{-3}$  and $\mathcal{E}_{-7}$  respectively.}
\label{Triangle37}
\end{figure}

For the previous two theorems we can see that the basic pattern of the lattice are similar, i.e. square and rectangles.  Figure \ref{Square12}  shows that the exceptional sets are the same basic shape as well.  Now we address the rings that have a triangular lattice.  We define $\omega = (1 \slash 2) + ( \sqrt{-3}  \slash 2)$.  For $a+b\omega \in \mathbb{Z}[\omega]$, $|a+b\omega|^{2} = a^{2} + ab + b^{2}$.  Let $\mathcal{E}_{-3} = \{ n \in \mathbb{Z}[\omega] : |n|^{2} \leq 1 \}$. 

\begin{theorem}
There exists a decomposition similar to (\ref{esconj}) for every element in $\mathbb{D}[-3] \backslash \mathcal{E}_{-3}$.
\end{theorem}

\begin{proof}
\ 

\noindent If we divide the denominator by $4$, proposition \ref{supporting6} tells us that we have sixteen possible remainders.  All associates of $2$, which is a prime number in this ring, are explained by the following decomposition $$ \frac{4}{2} = \frac{1}{1} + \frac{1}{2} + \frac{1}{2}. $$

\ 

\noindent Accounting for this exception, we see that four remainders do not generate prime numbers.  These remainders are $0, 2\omega, 2, 2 + 2\omega$. \\

\noindent Due to propositions \ref{supporting1}  and \ref{supporting2} we see that of the left-over remainders it suffices to find a decomposition for remainders $-1$ and $-(1 +\omega)$.  This is because six remainders are associates of $-1$  and six remainders are associates of $-(1+\omega)$. \\

\noindent Proposition \ref{supporting5} finds a decomposition when the remainder is $-1$. \\

\noindent Because $|1+\omega|^{2} = 3$  and $(1,3) = 1$, we can use proposition \ref{supporting3} to suggest that for $x+y\omega \in \mathbb{Z}[\omega]$  there exists $c+d\omega \in \mathbb{Z}[\omega]$  and $r \in \{ -1, 0 , 1 \}$  so that $x+y\omega = (1+ \omega)(c+d\omega)+r$.  Let $q_{1+\omega} \left( (c+d\omega), r \right) = (1+\omega)(c+d\omega) + r.$ \\

\noindent Notice that $p \left( q_{1+\omega} \left( (c+d\omega), 0 \right), -(1+\omega) \right)$  is not a prime number except when $c+d\omega = 0$, but in this case $$\frac{4}{1+\omega} = \frac{1}{1} + \frac{1}{\omega} + \frac{1}{1+\omega}.$$

\ 

\noindent Proposition \ref{supporting4} tells us that it suffices to find a solution for $r=1$  and proposition \ref{supporting5}  finds a decomposition when $r=1$.
\end{proof}
\ 

We now let $\omega = (1 \slash 2) + ( \sqrt{-7}  \slash 2)$.  For $a+b\omega \in \mathbb{Z}[\omega]$, $|a+b\omega|^{2} = a^{2} + ab + 2b^{2}$.  Let $\mathcal{E}_{-7} = \{ n \in \mathbb{Z}[\omega] : |n|^{2} \leq 2 \}$. 

\begin{theorem}
There exists a decomposition similar to (\ref{esconj}) for every element in $\mathbb{D}[-7] \backslash \mathcal{E}_{-7}$.
\end{theorem}

\begin{proof}
\ 

\noindent If we divide the denominator by $4$, proposition \ref{supporting6}  tells us that we have sixteen possible remainders.  Because all associates of the prime divisors of $2$  are elements of $\mathcal{E}_{-7}$, we see that twelve remainders do not generate prime numbers.  These remainders are $0, \omega, 2\omega, -\omega, 1+ \omega, 1 - \omega, 2, -2 + \omega, 2 + 2\omega, 2 - \omega, -1 + \omega, -1 - \omega$. \\

\noindent Due to propositions \ref{supporting1}  and \ref{supporting2}  we see that of the remaining scenarios it suffices to find a decomposition for remainders $-1$  and $-(1 - 2\omega)$.  This is because two remainders are associates of $-1$  and two remainders are associates of $-(1-2\omega)$. \\

\noindent Proposition \ref{supporting5}  finds a decomposition when the remainder is $-1$. \\

\noindent Because $|1-2\omega|^{2} = 7$  and $(2,7)=1$  we can use proposition \ref{supporting3}  to suggest that for $x+y\omega \in \mathbb{Z}[\omega]$  there exists $c+d\omega \in \mathbb{Z}[\omega]$  and $r \in \{ -3, -2, -1, 0 , 1, 2, 3 \}$  so that $x+y\omega = (1-2\omega)(c+d\omega) + r$.  Let $q_{1-2\omega} \left( (c+d\omega), r \right) = (1-2\omega)(c+d\omega) + r.$ \\

\noindent Notice that $p \left( q_{1-2\omega} \left( (c+d\omega), 0 \right), -(1-2\omega) \right)$  is not a prime number except when $c+d\omega = 0$, but in this case $$\frac{4}{-1+2\omega} = \frac{1}{\omega} + \frac{1}{-1+\omega} + \frac{1}{-2+4\omega}.$$ 

\ 

\noindent Proposition \ref{supporting4}  tells us that it suffices to find a solution for $r \in \{ 1, 2, 3 \}$. \\

\noindent Proposition \ref{supporting5}  finds a decomposition when $r=1$. \\

\noindent Define 

\begin{align*}
s_{1-2\omega,2} \left( q_{1-2\omega} \left( (c+d\omega), 2 \right) \right) &= \frac{p \left( q_{\omega} \left( (c+d\omega), 2 \right), -(1-2\omega) \right) - 1}{1-2\omega} \\
&= 4(c+d\omega) +(-2 + 2\omega) 
\end{align*}
\begin{align*}
s_{1-2\omega,3} &\left( q_{1-2\omega} \left( (c+d\omega), 3 \right) \right) \\
&= \frac{p \left( q_{1-2\omega} \left( (c+d\omega), 3 \right), -(1-2\omega) \right) \cdot q_{1-2\omega} \left( (c+d\omega), 3 \right) - 1}{1-2\omega} \\
&= (c+d\omega) \left( 4 \left( (1-2\omega)(c+d\omega) +6 \right) - (1-2\omega) \right)- (8 - 10\omega).
\end{align*}

\ 

\noindent If we write 
\begin{align*}
p &= p \left( q_{1-2\omega} \left( (c+d\omega), 2 \right), -(1-2\omega) \right) \\
q_{1-2\omega} &= q_{1-2\omega} \left( (c+d\omega), 2 \right) \\  
s_{1-2\omega,2} &= s_{1-2\omega,2} \left( q_{1-2\omega} \left( (c+d\omega), 2 \right) \right)
\end{align*}

\ 

\noindent we see that 
\begin{align*} 
\frac{4}{p} &= \frac{1}{q_{1-2\omega}} + \frac{1-2\omega}{p \cdot q_{1-2\omega}} \\
&= \frac{1}{q_{1-2\omega}} + \frac{(1-2\omega) \cdot s_{1-2\omega,2} }{p \cdot q_{1-2\omega} \cdot s_{1-2\omega,2}} \\
&= \frac{1}{q_{1-2\omega}} + \frac{p - 1}{p \cdot q_{1-2\omega} \cdot s_{1-2\omega,2}} \\
&= \frac{1}{q_{1-2\omega}} + \frac{1}{q_{1-2\omega} \cdot s_{1-2\omega,2}} - \frac{1}{p \cdot q_{1-2\omega} \cdot s_{1-2\omega,2}}.
\end{align*}

\ 

\noindent If we write 
\begin{align*}
p &= p \left( q_{1-2\omega} \left( (c+d\omega), 3 \right), -(1-2\omega) \right) \\
q_{1-2\omega} &= q_{1-2\omega} \left( (c+d\omega), 3 \right) \\  
s_{1-2\omega, 3} &= s_{1-2\omega, 3} \left( q_{1-2\omega} \left( (c+d\omega), 3 \right) \right)
\end{align*}

\ 

\noindent we see that 
\begin{align*} 
\frac{4}{p} &= \frac{1}{q_{1-2\omega}} + \frac{1-2\omega}{p \cdot q_{1-2\omega}} \\
&= \frac{1}{q_{1-2\omega}} + \frac{(1-2\omega) \cdot s_{1-2\omega,3} }{p \cdot q_{1-2\omega} \cdot s_{1-2\omega,3}} \\
&= \frac{1}{q_{1-2\omega}} + \frac{p \cdot q_{1-2\omega} - 1}{p \cdot q_{1-2\omega} \cdot s_{1-2\omega,3}} \\
&= \frac{1}{q_{1-2\omega}} + \frac{1}{ s_{1-2\omega,3}} - \frac{1}{p \cdot q_{1-2\omega} \cdot s_{1-2\omega,3}}.
\end{align*}

\ 

\noindent We only need to mention the following decompositions because the primes of which the are products were not in $\mathcal{E}_{-7}$. \\

\begin{align*}
\frac{4}{\omega^{2}} &= \frac{1}{-1} + \frac{1}{-1+\omega} + \frac{1}{-1+\omega} \\
\frac{4}{(1-\omega)^{2}} &= \frac{1}{-1} + \frac{1}{-\omega} + \frac{1}{-\omega} \\
\frac{4}{\omega(1-\omega)} &= \frac{1}{1} + \frac{1}{2} + \frac{1}{2}.
\end{align*}

\end{proof}

\ 

\begin{figure} 
\vspace{-2cm}
\adjustbox{trim = {0.055\width} {0.0\height} {0.0\width} {0.0\height}, clip}{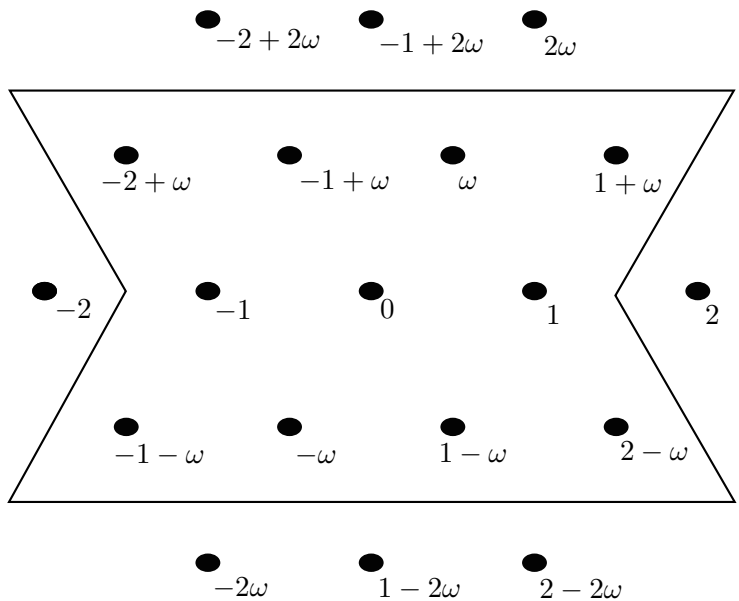}
\vspace{-18.5cm}
\caption{Though it is not to scale, this represents the lattice for $\mathbb{Z}[\omega]$  where $\omega = (1 \slash 2) + ( \sqrt{-11} \slash 2)$.  Lattice points inside the bounded region are precisely the points in $\mathcal{E}_{-11}$.}
\label{Triangle11}
\end{figure}

Much like the first two rings we considered, the previous two rings had a similar lattice and figure \ref{Triangle37} tells us that the exceptional sets have the same basic shape as well.  The next ring is unique in many ways.  We similarly define $\omega = (1 \slash 2) + ( \sqrt{-11}  \slash 2)$.  For $a+b\omega \in \mathbb{Z}[\omega]$, $|a+b\omega|^{2} = a^{2} + ab + 3b^{2}$.  Let $\mathcal{E}_{-11} = \{ n \in \mathbb{Z}[\omega] : |n|^{2} \leq 5 \} \backslash \{ 2, -2\}$. 

\begin{theorem}
There exists a decomposition similar to (\ref{esconj}) for every element in $\mathbb{D}[-11] \backslash \mathcal{E}_{-11}$.
\end{theorem}

\begin{proof}
\ 

\noindent If we divide the denominator by $4$, proposition \ref{supporting6}  tells us that we have sixteen possible remainders.  All associates of $2$, which is a prime number in this ring, are explained by the following decomposition $$ \frac{4}{2} = \frac{1}{1} + \frac{1}{2} + \frac{1}{2}.$$

\ 

\noindent Accounting for this exception, we see that four remainders do not generate prime numbers.  These remainders are $0, 2\omega, 2, 2 + 2\omega$.  \\

\noindent Due to propositions \ref{supporting1}  and \ref{supporting2} we see that of the remaining scenarios it suffices to find a decomposition for remainders $-1,-\omega,-(1+\omega),-(1 + 2\omega)$.  This is because two remainders are associates of $-1$, two remainders are associates of $-\omega$, two remainders are associates of the conjugate of $-\omega$, two remainders are associates of $-(1+\omega)$, two remainders are associates of the conjugate of $-(1+\omega)$  and two remainders are associates of $-(1+2\omega)$.  \\

\noindent Proposition \ref{supporting5}  finds a decomposition when the remainder is $-1$. \\ 

\noindent We first find a decomposition for remainder $-\omega$.  Because $|\omega|^{2}=3$  and $(1,3)=1$, we can use proposition \ref{supporting3}  to suggest that for $x+y\omega \in \mathbb{Z}[\omega]$  there exists $c+d\omega \in \mathbb{Z}[\omega]$  and $r \in \{ -1, 0 , 1 \}$  so that $x+y\omega = \omega(c+d\omega) + r$.  Let $q_{\omega} \left( (c+d\omega), r \right) = \omega(c+d\omega) + r.$ \\

\noindent Notice that $p \left( q_{\omega} \left( (c+d\omega), 0 \right), -\omega \right)$  is not a prime number except when $c+d\omega = 0$, but this prime number is an element of $\mathcal{E}_{-11}$. \\

\noindent Proposition \ref{supporting4}  tells us that it suffices to find a solution for $r=1$  and proposition \ref{supporting5}  finds a decomposition when $r=1$. \\

\noindent Next we find a decomposition for remainder $-(1+\omega)$.  Because $|1+\omega|^{2} = 5$  and $(1,5) = 1$,  we can use proposition \ref{supporting3}  again to suggest that for $x+y\omega \in \mathbb{Z}[\omega]$  there exists $c+d\omega \in \mathbb{Z}[\omega]$  and $r \in \{ -2, -1, 0 , 1, 2 \}$  so that $x+y\omega = (1+\omega)(c+d\omega)+r$.  Let $q_{1+\omega} \left( (c+d\omega), r \right) = (1+\omega)(c+d\omega) + r.$ \\

\noindent We see that $p \left( q_{1+\omega} \left( (c+d\omega), 0 \right), -(1+\omega) \right)$  is not a prime number, except when $c+d\omega = 0$, but this prime number is an element of $\mathcal{E}_{-11}$. \\

\noindent Proposition \ref{supporting4}  tells us that it suffices to find a solution when $r \in \{ 1, 2 \}$  and proposition \ref{supporting5}  finds a decomposition when $r=1$. \\

\noindent Define 

\begin{align*}
s_{1+\omega,2} &\left( q_{1+\omega} \left( (c+d\omega), 2 \right) \right) \\
&= \frac{p \left( q_{1+\omega} \left( (c+d\omega), 2 \right), -(1+\omega) \right) \cdot q_{1+\omega} \left( (c+d\omega),2 \right) - 1}{1+\omega} \\
&= (c+d\omega) \left( 4 \left( (1+\omega)(c+d\omega)+4 \right) - (1 + \omega) \right) +(4 + 3\omega).
\end{align*}

\ 

\noindent If we write 
\begin{align*}
p &= p \left( q_{1+\omega} \left( (c+d\omega), 2 \right), -(1+\omega) \right) \\
q_{1+\omega} &= q_{1+\omega} \left( (c+d\omega), 2 \right) \\  
s_{1+\omega,2} &= s_{1+\omega,2} \left( q_{1+\omega} \left( (c+d\omega), 2 \right) \right)
\end{align*}

\ 

\noindent we see that 
\begin{align*} 
\frac{4}{p} &= \frac{1}{q_{1+\omega}} + \frac{1+\omega}{p \cdot q_{1+\omega}} \\
&= \frac{1}{q_{1+\omega}} + \frac{(1+\omega) \cdot s_{1+\omega,2} }{p \cdot q_{1+\omega} \cdot s_{1+\omega,2}} \\
&= \frac{1}{q_{1+\omega}} + \frac{p \cdot q_{1+\omega} - 1}{p \cdot q_{1+\omega} \cdot s_{1+\omega,2}} \\
&= \frac{1}{q_{1+\omega}} + \frac{1}{s_{1+\omega,2}} - \frac{1}{p \cdot q_{1+\omega} \cdot s_{1+\omega,2}}.
\end{align*}

\ 

\noindent Finally we find a decomposition for $-(1+2\omega)$.  Because $|1+2\omega|^{2}=15$  and $(2,15)=1$, we can use proposition \ref{supporting3}  to suggest that for $x+y\omega \in \mathbb{Z}[\omega]$  there exists $c+d\omega \in \mathbb{Z}[\omega]$  and $r \in \{ -7, -6, -5, -4, -3, -2, -1, 0 , 1, 2, 3, 4, 5, 6, 7 \}$  so that $x+y\omega = (1+2\omega)(c+d\omega) + r$.  Let $q_{1+2\omega} \left( (c+d\omega), r \right) = (1+2\omega)(c+d\omega) + r.$ \\

\noindent Because $(1+2\omega) = -1 \cdot (1-\omega)(2 - \omega)$, $3 = \omega(1-\omega) $  and $5 = (1+\omega)(2-\omega)$  we see that $p \left( q_{1+2\omega} \left( (c+d\omega), 0 \right), -(1+2\omega) \right)$, $p \left( q_{1+2\omega} \left( (c+d\omega), \pm 3 \right), -(1+2\omega) \right)$, \sloppy $p \left( q_{1+2\omega} \left( (c+d\omega), \pm 5 \right), -(1+2\omega) \right)$  and $p \left( q_{1+2\omega} \left( (c+d\omega), \pm 6 \right), -(1+2\omega) \right)$  are not prime numbers. \\

\noindent Proposition \ref{supporting4}  tells us that it suffices to find a solution for $r \in \{ 1, 2, 4, 7 \}$  and proposition \ref{supporting5} finds a decomposition when $r=1$. \\

\noindent Define 

\begin{align*}
s_{1+2\omega,2} &\left( q_{1+2\omega} \left( (c+d\omega), 2 \right) \right) \\
&= \frac{p \left( q_{1+2\omega} \left( (c+d\omega), 2 \right), -(1+2\omega) \right) \cdot q_{1+2\omega} \left( (c+d\omega),2 \right) - 1}{1+2\omega} \\
&= (c+d\omega) \left( 4 \left( (1+2\omega)(c+d\omega)+4 \right) - (1+2\omega) \right) +(1 - 2\omega) 
\end{align*}
\begin{align*}
s_{1+2\omega,4} \left( q_{1+2\omega} \left( (c+d\omega), 4 \right) \right) &= \frac{p \left( q_{1+2\omega} \left( (c+d\omega), 4 \right), -(1+2\omega) \right) - 1}{1+2\omega} \\
&= 4(c+d\omega) +(2-2\omega) 
\end{align*}
\begin{align*}
s_{1+2\omega,7} &\left( q_{1+2\omega} \left( (c+d\omega), 7 \right) \right) \\
&= \frac{p \left( q_{1+2\omega} \left( (c+d\omega), 7 \right), -(1+2\omega) \right) \cdot q_{1+2\omega} \left( (c+d\omega),7 \right) - 1}{1+2\omega} \\
&= (c+d\omega) \left( 4 \left( (1+2\omega)(c+d\omega)+8 \right) - (1 +2\omega) \right) +(35  - 26\omega).
\end{align*}

\ 

\noindent If we write 
\begin{align*}
p &= p \left( q_{1+2\omega} \left( (c+d\omega), 2 \right), -(1+2\omega) \right) \\
q_{1+2\omega} &= q_{1+2\omega} \left( (c+d\omega), 2 \right) \\  
s_{1+2\omega,2} &= s_{1+2\omega,2} \left( q_{1+2\omega} \left( (c+d\omega), 2 \right) \right)
\end{align*}

\ 

\noindent we see that 
\begin{align*} 
\frac{4}{p} &= \frac{1}{q_{1+2\omega}} + \frac{1+2\omega}{p \cdot q_{1+2\omega}} \\
&= \frac{1}{q_{1+2\omega}} + \frac{(1+2\omega) \cdot s_{1+2\omega,2} }{p \cdot q_{1+2\omega} \cdot s_{1+2\omega,2}} \\
&= \frac{1}{q_{1+2\omega}} + \frac{p \cdot q_{1+2\omega} - 1}{p \cdot q_{1+2\omega} \cdot s_{1+2\omega,2}} \\
&= \frac{1}{q_{1+2\omega}} + \frac{1}{s_{1+2\omega,2}} - \frac{1}{p \cdot q_{1+2\omega} \cdot s_{1+2\omega,2}}.
\end{align*}

\ 

\noindent If we write 
\begin{align*}
p &= p \left( q_{1+2\omega} \left( (c+d\omega), 4 \right), -(1+2\omega) \right) \\
q_{1+2\omega} &= q_{1+2\omega} \left( (c+d\omega), 4 \right) \\  
s_{1+2\omega,4} &= s_{1+2\omega,4} \left( q_{1+2\omega} \left( (c+d\omega), 4 \right) \right)
\end{align*}

\ 

\noindent we see that 
\begin{align*} 
\frac{4}{p} &= \frac{1}{q_{1+2\omega}} + \frac{1+2\omega}{p \cdot q_{1+2\omega}} \\
&= \frac{1}{q_{1+2\omega}} + \frac{(1+2\omega) \cdot s_{1+2\omega,4} }{p \cdot q_{1+2\omega} \cdot s_{1+2\omega,4}} \\
&= \frac{1}{q_{1+2\omega}} + \frac{p - 1}{p \cdot q_{1+2\omega} \cdot s_{1+2\omega,4}} \\
&= \frac{1}{q_{1+2\omega}} + \frac{1}{q_{1+2\omega} \cdot s_{1+2\omega,4}} - \frac{1}{p \cdot q_{1+2\omega} \cdot s_{1+2\omega,4}}.
\end{align*}

\ 

\noindent If we write 
\begin{align*}
p &= p \left( q_{1+2\omega} \left( (c+d\omega), 7 \right), -(1+2\omega) \right) \\
q_{1+2\omega} &= q_{1+2\omega} \left( (c+d\omega), 7 \right) \\  
s_{1+2\omega,7} &= s_{1+2\omega,7} \left( q_{1+2\omega} \left( (c+d\omega), 7 \right) \right)
\end{align*}

\ 

\noindent we see that 
\begin{align*} 
\frac{4}{p} &= \frac{1}{q_{1+2\omega}} + \frac{1+2\omega}{p \cdot q_{1+2\omega}} \\
&= \frac{1}{q_{1+2\omega}} + \frac{(1+2\omega) \cdot s_{1+2\omega,7} }{p \cdot q_{1+2\omega} \cdot s_{1+2\omega,7}} \\
&= \frac{1}{q_{1+2\omega}} + \frac{p \cdot q_{1+2\omega} - 1}{p \cdot q_{1+2\omega} \cdot s_{1+2\omega,7}} \\
&= \frac{1}{q_{1+2\omega}} + \frac{1}{s_{1+2\omega,7}} - \frac{1}{p \cdot q_{1+2\omega} \cdot s_{1+2\omega,7}}.
\end{align*}

\ 

\noindent We only mention the following decompositions because the primes of which they are products were not in $\mathcal{E}_{-11}$:

\begin{align*}
\frac{4}{\omega^{2}} &= \frac{1}{-1} + \frac{1}{\omega^{2}} + \frac{1}{\omega} \\
\frac{4}{(1+\omega)^{2}} &= \frac{1}{-1+\omega} + \frac{1}{2} + \frac{1}{12-18\omega} \\
\frac{4}{\omega(1-\omega)} &= \frac{1}{2} + \frac{1}{2} + \frac{1}{3} \\
\frac{4}{\omega(1+\omega)} &= \frac{1}{-1} + \frac{1}{\omega} + \frac{1}{1+\omega} \\
\frac{4}{(1-\omega)(1+\omega)} &= \frac{1}{1} + \frac{1}{-\omega} + \frac{1}{3+3\omega} \\
\frac{4}{(1+\omega)(2-\omega)} &= \frac{1}{2} + \frac{1}{4} + \frac{1}{20}.
\end{align*}
\end{proof}

We can see from figure \ref{Triangle11} that the exceptional set has a different basic shape from the shape of the exceptional sets from the other triangular lattices.  Recall that we also conjectured that there are some unique factorization ring on integers for quadratic fields for which (\ref{esconj}) has a decomposition outside of an exceptional set.  We would also like to conjecture that the exceptional sets will be connected in a discrete topology imposed on these rings.  Answering many of the conjectures in this paper will be the next step in this procedure and we strongly feel that understanding the Erd\H{o}s-Straus conjecture in these new contexts will illuminate the fundamental problems with original conjecture.

\end{document}